\documentclass[11pt,reqno]{amsart}
\usepackage[foot]{amsaddr}

\newcommand{\Cnlr}{C^{(r)}_{n, \ell}}
\newcommand{\C}{C_1}
\newcommand{\CC}{C_2}

\title{Exact threshold and limiting distribution for non-linear Hamilton cycles}

\author{Byron Chin}
\address{Department of Mathematics, Massachusetts Institute of Technology}
\email{byronc@mit.edu}

\numberwithin{equation}{section}

\usepackage{arxiv}

\begin{document}

\begin{abstract}
    For positive integers $r > \ell \geq 1$, an $\ell$-cycle in an $r$-uniform hypergraph is a cycle where each edge consists of $r$ vertices and each pair of consecutive edges intersect in $\ell$ vertices. For $\ell \geq 2$, we determine the limiting distribution of the number of Hamilton $\ell$-cycles in an Erd\H{o}s--R\'enyi random hypergraph. The behavior is distinguished in two cases:
    \begin{itemize}
        \item[--] When $\ell \geq 3$, the number of cycles concentrates when the expectation diverges and converges to a Poisson distribution when the expectation is constant.
        \item[--] When $\ell = 2$, the normalized number of cycles converges to a lognormal distribution when the expectation diverges and converges to a lognormal mixture of Poisson distributions when the expectation is constant.
    \end{itemize}
    As a result we pin down the exact threshold for the appearance of non-linear Hamilton cycles in random hypergraphs, confirming a conjecture of Narayanan and Schacht. 
\end{abstract}

\maketitle

\section{Introduction}
The problem of determining the critical density at which a specific substructure will appear inside a random structure has been central to the development of probabilistic combinatorics. Starting with the case of fixed size subgraphs in random graphs, a fairly complete answer was determined by Bollob\'as using the second moment method \cite{B:81}. The cases of spanning structures have been more difficult to handle, with a general partial answer provided by Riordan \cite{R:00}, also by a careful second moment calculation. A breakthrough work of Park and Pham \cite{PP:24} resolved the Kahn--Kalai conjecture, which gives a general purpose tool to locate these thresholds up to logarithmic factors. However, pinning down the sharp location of thresholds still remains a challenge. 

Several substructures of interest have also been studied specifically. One of the most notable cases is that of perfect matchings. Early works of Erd\H{o}s and R\'enyi \cite{ER:66} determined the threshold for the appearance of a perfect matching in a random graphs. However, the corresponding problem for hypergraphs, known as Shamir's problem, was notoriously difficult before it was resolved in a breakthrough work of Johansson, Kahn, and Vu \cite{JKV:08}. 

We study the problem of determining when a random hypergraph contains another spanning structure of interest, the Hamilton cycle -- that is a cycle that uses every vertex of the graph. This falls into a long line of work searching for spanning structures in random (hyper)graphs. Starting with the case of graphs, foundational works by P\'{o}sa \cite{P:76}, Koml\'os and Szemer\'edi \cite{KS:83}, and Ajtai, Koml\'os, and Szemer\'edi \cite{AKS:85} pinned down not only the sharp threshold, but also the more precise hitting time result equating the appearance of a Hamilton cycle with the time at which the minimum degree becomes two. 

In the setting of hypergraphs, there is more than one way to form a Hamilton cycle. In particular, for any positive integers $r > \ell \geq 1$, we can consider the $r$-uniform $\ell$-cycle, where each edge consists of $r$ consecutive vertices and any two consecutive edges overlap in exactly $\ell$ vertices. The case when $\ell = 1$ is known as a linear, or loose, cycle whereas when $\ell = r-1$ the cycle is called tight. Building on work of Dudek and Frieze \cite{DF:11, DF:13}, Narayanan and Schacht \cite{NS:20} showed that the first moment threshold is sharp for the appearance of any non-linear Hamilton cycle. Their work is a combination of a careful second moment estimate with a general theorem of Friedgut \cite{F:99} characterizing sharp thresholds. 

However, finding the sharp threshold in this setting leaves more to be desired, as any density a factor of $1+\epsilon$ above the first moment threshold produces \textit{exponentially} many Hamilton cycles on average. This leaves the question of the appearance of Hamilton cycles when the expectation grows at a slower than exponential rate. Narayanan and Schacht conjectured that the first moment threshold is much sharper than they were able to show. In particular, a Hamilton cycle should appear as soon as the expected number tends to infinity, at any arbitrarily slow rate.

The motivation of this paper is to confirm this conjecture by carrying out a further refined second moment calculation. As we will see, we will be able to say much more about the number of Hamilton cycles that appear. While for $\ell \geq 3$, a careful analysis reveals a sharp second moment, the special case of $\ell = 2$ requires using a technique known as small subgraph conditioning. In this case we show that the second moment method is sharp not for the number of Hamilton 2-cycles, but for a suitably modified random variable. A natural consequence of this method is the characterization that the number of Hamilton 2-cycles converges to an explicit lognormal distribution. Using the above characterization we are able to deduce the limiting behavior of the number of Hamilton cycles in the constant expectation regime as well. 

\subsection{Main definitions}\label{subsec: MD}
Let $G_r(n, p)$ be the $r$-uniform hypergraph where each subset of the vertices of size $r$ is included as an edge independently with probability $p$, and let $\Cnlr$ be the $r$-uniform $\ell$-cycle on $n$ vertices. For any hypergraph $H$, we will use $Z(H)$ to denote the random variable counting the number of copies of $H$ in $G_r(n,p)$ and $N_H$ to denote the number of copies of $H$ in the complete hypergraph. In particular, $N_H = \frac{(n)_{v(H)}}{\mathrm{Aut}(H)}$ where $(n)_k = n(n-1)\cdots (n-k+1)$ is the falling factorial and $\mathrm{Aut}(H)$ is the number of automorphisms of $H$. 

The first moment threshold for $\Cnlr$ can be described as follows. Let $s = r-\ell$ and $1 \leq t \leq s$ be the unique integer satisfying $t \equiv r \pmod s$. Define $\lambda(r, \ell) = t!(s-t)!$, then the first moment threshold, where $\E{Z(\Cnlr)} = 1$, is of the form
\[ p^*(r,\ell) = (1+o(1))\frac{\lambda(r,\ell)e^{r-\ell}}{n^{r-\ell}}. \]
For the remainder of the paper, we will parameterize $p = cp^*(r,\ell)$. We will focus on the case where $c$ is bounded, in particular $c$ around 1, as this captures the setting where the behavior of the threshold is not known. 

\subsection{Results}
With these definitions, we can more precisely describe the current knowledge on this problem. The state of the art was proven by Narayanan and Schacht, who showed that $p^*(r,\ell)$ is a sharp threshold for the appearance of $\Cnlr$. However, note that the probability that a fixed copy of $\Cnlr$ appears is $p^{\frac{n}{r-\ell}}$, so the expected number of copies of $\Cnlr$ for $p > (1+\epsilon)p^*(r,\ell)$ is at least $(1+\epsilon)^{\frac{n}{r-\ell}}$, which is exponentially large. This motivates their following conjecture.
\begin{conjecture}[{\cite[Conjecture 4.1]{NS:20}}]\label{conj: NS}
    For all integers $r > \ell > 1$, as $n \to \infty$ with $(r-\ell) \ | \ n$, if $p = p(n)$ is such that $\E{Z(\Cnlr)} \to \infty$ then 
    \[ \Prob{\Cnlr \subset G_r(n, p)} \to 1. \]
\end{conjecture}

Our main result is that for $\ell \geq 3$, the number of Hamilton $\ell$-cycles concentrates well around its expectation, whereas for $\ell = 2$, the random variable $\frac{Z(\Cnlr)}{\E{Z(\Cnlr)}}$ has a log-normal limiting distribution whenever the expectation tends to infinity. 
\begin{theorem}\label{thm: lognormal}
    For all integers $r > \ell > 1$, as $n \to \infty$ with $(r-\ell) \ | \ n$, if $p = (c+o(1))p^*(r,\ell)$ is such that $\E{Z(\Cnlr)} \to \infty$, then 
    \begin{enumerate}[label=\arabic*.]
        \item\label{thm: lognormal1} For $\ell \geq 3,$ 
        \[ \frac{Z(\Cnlr}{\E{Z(\Cnlr)}} = 1+o(1) \quad w.h.p. \]
        \item\label{thm: lognormal2} For $\ell = 2, $
        \[\frac{Z(\Cnlr)}{\E{Z(\Cnlr)}} \overset{\mathrm{d}}{\longrightarrow} \mathrm{Lognormal}\left(-\sum_{k=1}^{\infty} \frac{A_k c^{-k}e^{-k(r-\ell)}}{2(r-\ell)^2}, \sum_{k=1}^\infty \frac{A_k c^{-k}e^{-k(r-\ell)}}{(r-\ell)^2}\right). \]
    \end{enumerate}
\end{theorem}
Here, $A_k$ are explicit constants whose definition we postpone to (\ref{eq: Ak}). We note that $A_k$ remains bounded as $k \to \infty$, so the infinite series converges quickly to a constant. As a corollary, we confirm the above conjecture.
\begin{corollary}\label{cor: conj}
    Conjecture \ref{conj: NS} is true.
\end{corollary}
\begin{proof}
    For $\ell \geq 3$, the result follows immediately. For $\ell = 2$, since the log-normal distribution has a density with respect to the positive real line, we have 
    \[ \Prob{Z(\Cnlr) = 0} \leq \Prob{\frac{Z(\Cnlr)}{\E{Z(\Cnlr)}} = 0} \xrightarrow{n \to \infty} 0. \]
\end{proof}
By performing a subsampling argument, we can also deduce the limiting behavior when the expectation is constant, which is the content of our final theorem.
\begin{theorem}\label{thm: poisson}
    For all integers $r > \ell > 1$, as $n \to \infty$ with $(r-\ell) \ \vert \ n$, if $p = p(n)$ is such that $\E{Z(\Cnlr)} = m$, then
    \begin{enumerate}
        \item[1.] For $\ell \geq 3$, 
        \[ Z(\Cnlr) \overset{\mathrm{d}}{\longrightarrow} \mathrm{Pois}(m). \]
        \item[2.] For $\ell = 2$, 
        \[ Z(\Cnlr) \overset{\mathrm{d}}{\longrightarrow} \mathrm{Pois}(mL) \]
        where $L \sim \mathrm{Lognormal}\left(-\sum_{k=1}^\infty \frac{A_ke^{-k(r-\ell)}}{2(r-\ell)^2}, \sum_{k=1}^\infty \frac{A_ke^{-k(r-\ell)}}{(r-\ell)^2}\right)$.
    \end{enumerate}
\end{theorem}
Notice that a routine application of Markov's inequality implies that $Z(\Cnlr) = 0$ with high probability when $\E{Z(\Cnlr)} \to 0$. Thus, our theorems characterize the limiting behavior of $Z(\Cnlr)$ in all regimes.

\subsection{Proof overview}
The difficulty with determining the sharp threshold for Hamilton $\ell$-cycles is that the standard second moment does not always work. Narayanan and Schacht \cite{NS:20} overcome this difficulty by estimating the second moment up to a constant factor, and applying a general purpose tool of Friedgut \cite{F:99} to show that the threshold must be sharp. A key observation of this paper, which is implicit in their handling of the second moment, is that the additional constant factor in the second moment arises from the fluctuations of a fixed collection of small subgraphs. If one were able to isolate the contribution from these small subgraphs, the second moment estimate can be made sharp. 

The tool to treat the contribution from these small subgraphs and handle situations where the second moment is off by a constant factor is known as small subgraph conditioning. The idea was introduced in a breakthrough work of Robinson and Wormald \cite{RW:94} to show that random regular graphs have Hamilton cycles with high probability. Janson \cite{J:95} realized the applicability of the technique and presented a general method to establish contiguity between null and planted distributions. The small subgraph conditioning method went on to play an important role in the analysis of many statistical physics models \cite{GSV:16, CEH:16, MWW:09, MNS:15, CEJKK:18}. On the other hand, applications to more combinatorial settings have been less frequent.

An important commonality between each of the above applications is a sparsity of the model, i.e. the number of edges is proportional to the number of vertices. Notable works of Banerjee \cite{B:18} and Abbe, Li, and Sly \cite{ALS:22} were able to implement a version of the small subgraph conditioning technique in dense models, namely the stochastic block model and symmetric perceptron respectively. 

We rely on a similar application of small subgraph conditioning in the dense setting to explain the additional constant factor in the second moment. In particular, we analyze the random variable $X = \frac{Z(\Cnlr)}{\E{Z(\Cnlr)}}\exp(-Y)$ where 
\[Y = \sum_{k=1}^{\log n} t_kY(P_k).\]
Here $P_k$ is $r$-uniform $\ell$-path with $k$ edges, and $Y(H)$ is the normalized subgraph count of $H$ in a (random) graph $G$ defined by 
\[ Y(H) = \frac{1}{\sqrt{N_H}}\sum_{H}\prod_{e \in H} \frac{\mathbbm{1}\{e \in G\} - p}{\sqrt{p(1-p)}} \]
where the sum runs over all copies of $H$ in the complete graph. Note that the normalization is chosen so that $\E{Y(H)} = 0 \text{ and } \Var{Y(H)} = 1.$

The intuition for this choice of random variable is that each $P_k$ can be viewed as a ``short segment'' of a Hamilton cycle. The appearance of each such short segment has a constant multiplicative effect on the expected number of Hamilton cycles. Thus, the total contribution heuristically takes an exponential form, and the correction $e^{-Y}$ removes this contribution. With foresight, we choose the coefficients 
\begin{equation}\label{eq: Ak}
    t_k = \frac{\sqrt{A_k c^{-k} e^{-k(r-\ell)}}}{(r-\ell)} \text{ where } A_k = \frac{\mathrm{Aut}(P_k)}{(t!(s-t)!)^k}.
\end{equation}
This is the precise quantitative version of the ``constant multiplicative effect'' alluded to above. Notice that for $k$ a sufficiently large constant in terms of $r$ and $\ell$, $A_k$ is constant (corresponding to the extra permutations arising from the two ends of the path). In particular, and importantly for our purposes, it is uniformly bounded by a function of $r$ and $\ell$.

\subsection{Organization}
In Section \ref{sec: SMC} we give a careful case analysis to get a sharp estimate on the second moment. This proves part 1 of Theorem \ref{thm: lognormal}. In Section \ref{sec: SSC} we quantify the contribution from the small subgraph conditioning modification of the random variable. Section \ref{sec: PMT} combines the two elements to prove part 2 of Theorem \ref{thm: lognormal}. Finally, we deduce Theorem \ref{thm: poisson} in Section \ref{sec: poisson}.

\subsection{Notation}
We use $\PP^*$ and $\EE^*$ to denote probability and expectation with respect to the measure $G_r(n,p)$ with a planted copy of $\Cnlr$. For the second moment, we use $\PP_t^{*2}$ and $\EE_t^{*2}$ to denote probability and expectation with respect to the measure $G_r(n,p)$ with two planted copies of $\Cnlr$, whose intersection consists of $t$ edges. We use $C_1$ and $C_2$ to enumerate over all copies of $\Cnlr$, and also $\C$ and $\CC$ to denote two independent uniformly random copies of $\Cnlr$ in the complete hypergraph. For ease of notation, sometimes we write $C$ in place of $\Cnlr$ in sub/superscripts.

\subsection*{Acknowledgements}
The author thanks Mehtaab Sawhney for helpful discussions and references. The author is supported by an NSF Graduate Research Fellowship. 

\section{Second moment computation}\label{sec: SMC}
In this section, we evaluate the second moment of $Z(\Cnlr)$. We know from \cite{NS:20} that this is simply $\E{Z(\Cnlr)}^2$ up to a constant factor. However, since we want to show a sharp second moment for $X$, we need to determine the exact constant factor. Thus, a more careful estimate of the overlap distribution is required. We begin by rewriting the second moment in terms of the equivalence classes of intersection subgraph:
\[ \sum_{t=1}^m \frac{\Prob{|\C \cap \CC| = t}}{p^t} = \sum_{F \subset \Cnlr} \frac{\Prob{\C \cap \CC \cong F}}{p^{e(F)}} \leq \sum_{F \subset \Cnlr} \frac{\E{N_F(\C \cap \CC)}}{p^{e(F)}}. \]
We will handle the various subgraphs in a few cases:
\begin{itemize}
    \item[$\boldsymbol{\cdot}$] Lemma \ref{lemma: small overlap}: if $v(F) \leq \log n$ we get the main contribution from the short segments.
    \item[$\boldsymbol{\cdot}$] Lemma \ref{lemma: mid overlap}: if $\log n < v(F) < n$ the contribution is vanishing due to the discrepancy between $\left(\frac{n}{e}\right)^k$ and $(n)_k$. 
    \item[$\boldsymbol{\cdot}$] Lemma \ref{lemma: big overlap conn}: if $v(F) = n$ and $F$ is connected, then by comparison to $F = \Cnlr$ the contribution is bounded by $\E{Z(\Cnlr)}^{-1}$.
    \item[$\boldsymbol{\cdot}$] Lemma \ref{lemma: big overlap disconn}: if $v(F) = n$ and $F$ is disconnected, the contribution is vanishing due to the limited number of copies of $F$ in $\Cnlr$. 
\end{itemize}
Each $F$ in the sum can be decomposed into its connected components as follows. For each $k$, let there be $n_k$ components isomorphic to $P_k$ in $F$. Denote the remaining isomorphism classes of components by $F_1, \ldots, F_a$, with multiplicities $m_1, \ldots, m_a$. Then 
\begin{equation}\label{eq: decomp}
    \frac{\E{N_F(\C \cap \CC)}}{p^{e(F)}} \leq \frac{(\frac{n}{r-\ell})^{2\sum_k n_k + \sum_j m_j}}{(n)_{v(F)}} \cdot \prod_k \frac{\mathrm{Aut}(P_k)^{n_k}}{n_k!}p^{-kn_k} \cdot \prod_{j=1}^a \frac{\mathrm{Aut}(F_j)^{m_j}}{m_j!}p^{-e(F_j)m_j}.
\end{equation}
We analyze this decomposition according to the cases listed above. 

\begin{lemma}\label{lemma: small overlap}
    \[ \sum_{F \subset \Cnlr, v(F) \leq \log n} \frac{\E{N_F(\C \cap \CC)}}{p^{e(F)}} \leq \exp((1+o(1))\sum_{k=1}^{\log n} \frac{A_k c^{-k} e^{-k(r-\ell)}}{(r-\ell)^2}). \]
\end{lemma}
\begin{proof}
    Note that we have $(n)_k = (1+o(1))n^k$ for $k = o(\sqrt{n})$. Thus, 
    \begin{align}
        \sum_{\substack{F \subset \Cnlr \\ v(F) \leq \log n}} \frac{\E{N_F(\C \cap \CC)}}{p^{e(F)}} &\leq (1+o(1))\frac{(\frac{n}{r-\ell})^{2\sum_k n_k + \sum_j m_j}}{n^{v(F)}} \cdot \prod_k \frac{\mathrm{Aut}(P_k)^{n_k}}{n_k!} \cdot \prod_{j=1}^a \frac{\mathrm{Aut}(F_j)^{m_j}}{m_j!} \cdot p^{-e(F)} \nonumber\\
        &\leq (1+o(1))\prod_{\substack{F \subset \Cnlr \\ v(F) \leq \log n \\ F \text{ conn.}}} \sum_{k=1}^\infty \frac{1}{k!}\left( \frac{n^2}{(r-\ell)^2n^{v(F)}} \mathrm{Aut}(F)p^{-e(F)}\right)^k \nonumber\\
        &= (1+o(1))\prod_{\substack{F \subset \Cnlr \\ v(F) \leq \log n \\ F \text{ conn.}}} \exp\left( \frac{n^2}{(r-\ell)^2n^{v(F)}} \mathrm{Aut}(F)p^{-e(F)}\right) \nonumber\\
        &= (1+o(1)) \exp(\sum_{\substack{F \subset \Cnlr \\ v(F) \leq \log n \\ F \text{ conn.}}} \frac{n^2}{(r-\ell)^2n^{v(F)}} \mathrm{Aut}(F)p^{-e(F)} ) \label{eq: poisson}
    \end{align}
    
    We first handle the terms corresponding to $F = P_k$. For this we explicitly have $e(P_k) = k$, $v(P_k) = \ell + k(r-\ell)$, $\mathrm{Aut}(P_k) = A_k(t!(s-t)!)^k$ where $A_k$ is the correction for the edges near the ends of the path. Importantly this is independent of the length of the path. Thus, we get 
    \begin{align*}
        \frac{n^2}{(r-\ell)^2n^{v(P_k)}}\mathrm{Aut}(P_k)p^{-e(P_k)} &= \frac{n^2}{n^{\ell + (r-\ell)k}} \cdot \frac{A_k}{(r-\ell)^2} (t!(s-t)!)^k \left( \frac{n^{r-\ell}}{ct!(s-t)!e^{r-\ell}} \right)^k \\
        &\leq \frac{A_k}{(r-\ell)^2} c^{-k}e^{-k(r-\ell)}n^{2-\ell}. 
    \end{align*}
    Now consider a connected component $F_j$ not isomorphic to any $P_k$. Suppose first that $v(F_j) < n$. Then $v(F_j) = v(P_{k_j})$ for some $k_j$. We compare the contribution of $F_j$ with that of $P_{k_j}$. The key observation is that we have the bound 
    \[ \mathrm{Aut}(F_j) \leq \mathrm{Aut}(P_{k_j}) \cdot (r!^2)^{k_j - e(F_j)}. \]
    This is because for each edge removed from $P_{k_j}$, we free up the vertices within two edges to be permuted, and there are at most $r!$ ways to permute within each of these edges. Thus we can bound
    \begin{align}
        \frac{n^2}{(r-\ell)^2n^{v(F_j)}}\mathrm{Aut}(F_j)p^{-e(F_j)} &\leq \frac{n^2}{n^{\ell + (r-\ell)k_j}} \cdot \frac{A_{k_j}}{(r-\ell)^2}(t!(s-t)!)^{k_j}(r!^2)^{k_j - e(F_j)}p^{k_j - e(F_j)}p^{-k_j} \nonumber\\
        &= \frac{n^2}{n^{\ell + (r-\ell)k_j}} \cdot \frac{A_{k_j}}{(r-\ell)^2}(t!(s-t)!)^{k_j}(r!^2p)^{k_j - e(F_j)}\left( \frac{n^{r-\ell}}{ct!(s-t)!e^{r-\ell}} \right)^{k_j} \nonumber\\
        &\leq \frac{A_{k_j}}{(r-\ell)^2}c^{-k_j}e^{-k_j(r-\ell)}n^{2-\ell} \left( \frac{r!^2ct!(s-t)!e^{r-\ell}}{n^{r-\ell}} \right)^{k_j - e(F_j)} \nonumber\\
        &= \frac{A_{k_j}}{(r-\ell)^2}c^{-k_j}e^{-k_j(r-\ell)}n^{2-\ell}n^{-(r-\ell+o(1))(k_j - e(F_j))} \label{eq: estimate}
    \end{align}
    Notice here that $e(F_j) < k_j$ by assumption, so the contribution is smaller by a factor of at least $n$ for every edge that is missing from $P_{k_j}$. In particular, we have the bound
    \begin{align}
        \sum_{\substack{F \subset \Cnlr\\ v(F) = v(P_{k_j})\\ F \text{ conn.}}} \frac{n^2}{(r-\ell)^2n^{v(F)}}\mathrm{Aut}(F)p^{-e(F)} &\leq \frac{A_{k_j}}{(r-\ell)^2}c^{-k_j}e^{-k_j(r-\ell)}n^{2-\ell} \sum_{k=0}^{k_j} \binom{k_j}{k}n^{-(r-\ell+o(1))k} \nonumber\\
        &= \frac{A_{k_j}}{(r-\ell)^2}c^{-k_j}e^{-k_j(r-\ell)}n^{2-\ell} \left( 1 + n^{-(r-\ell+o(1))} \right)^{k_j} \nonumber\\
        &\leq (1+o(1))\frac{A_{k_j}}{(r-\ell)^2}c^{-k_j}e^{-k_j(r-\ell)}n^{2-\ell} \label{eq: binomial}
    \end{align}
    for $k_j \leq \log n$. Thus, all together we have
    \[ \sum_{F \subset \Cnlr, v(F) \leq \log n} \frac{\E{N_F(\C \cap \CC)}}{p^{e(F)}} \leq \exp((1+o(1))n^{2-\ell}\sum_{k=1}^{\log n} \frac{A_k}{(r-\ell)^2}c^{-k}e^{-k(r-\ell)}). \]
\end{proof}
\begin{lemma}\label{lemma: mid overlap}
    \[ \sum_{\log n < v(F) \leq n-1}\frac{\E{N_F(\C \cap \CC)}}{p^{e(F)}} \lesssim_{r,\ell} \frac{1}{n}. \]
\end{lemma}
\begin{proof}
    For the remaining cases we require a more general purpose bound on the falling factorial in this regime. Stirling's approximation yields $(n)_k \geq n^k\exp(-n((1-\frac{k}{n})\log(1-\frac{k}{n}) + \frac{k}{n}))$. Applying this estimate in (\ref{eq: decomp}), the right hand side factors according to the connected components.
    \begin{gather*}
        \frac{\E{N_F(\C \cap \CC)}}{p^{e(F)}} \leq \exp((n-v(F))\log(1-\frac{v(F)}{n}) + v(F)) \prod_k \frac{1}{n_k!}\left( \frac{n^2}{n^{v(P_k)}}\mathrm{Aut}(P_k)p^{-e(P_k)} \right)^{n_k} \\
        \cdot \prod_{j=1}^a \frac{1}{m_j!}\left( \frac{n^2}{n^{v(F_j)}}\mathrm{Aut}(F_j)p^{-e(F_j)} \right)^{m_j}.
    \end{gather*}
    Note that we can express (\ref{eq: estimate}) as 
    \[ \frac{n^2}{n^{v(F_j)}}\mathrm{Aut}(F_j)p^{-e(F_j)} \leq A_{k_j}c^{-k_j}e^{-v(F_j)}e^\ell n^{-(r-\ell+o(1))(k_j-e(F_j))}. \]
    Substituting this into the above expression yields
    \begin{gather*}
        \frac{\E{N_F(\C \cap \CC)}}{p^{e(F)}} \leq \exp((n-v(F))\log(1-\frac{v(F)}{n})) \prod_k \frac{1}{n_k!}\left( A_k c^{-k} e^\ell n^{o(1)} \right)^{n_k} \\
        \cdot \prod_{j=1}^a \frac{1}{m_j!}\left( A_{k_j} c^{-k_j} e^\ell n^{-(r-\ell+o(1))(k_j-e(F_j))} \right)^{m_j}.
    \end{gather*}
    By the same calculation as (\ref{eq: poisson}), we have the estimate
    \begin{multline*}
        \sum_{v(F) = t} \frac{\E{N_F(\C \cap \CC)}}{p^{t}} \leq \exp((n-t)\log(1-\frac{t}{n})) \\ \exp(\sum_{F \subset \Cnlr, v(F) \leq t, F \text{ conn.}} A_{k_j} c^{-k_j} e^\ell n^{-(r-\ell+o(1))(k_j-e(F))})
    \end{multline*}
    The analog of (\ref{eq: binomial}) then gives 
    \begin{equation*}
        \sum_{v(F) = t}\frac{\E{N_F(\C \cap \CC)}}{p^{e(F)}} \leq \exp((n-t)\log(1-\frac{t}{n}))\exp((1+o(1))\sum_{k=1}^{t} A_k c^{-k} e^\ell).
    \end{equation*}
    The second exponential term is uniformly bounded by a constant in $r$ and $\ell$. Moreover, for $2\log n \leq t \leq n-2$, the first exponential term is at most $\frac{4}{n^2}$, while for $\log n \leq t \leq 2\log n$ or $t = n-1$ it is at most $\frac{1}{n}$. All together, 
    \[ \sum_{\log n < v(F) \leq n-1}\frac{\E{N_F(\C \cap \CC)}}{p^{e(F)}} \lesssim_{r,\ell} \frac{1}{n}. \]
\end{proof}
\begin{lemma}\label{lemma: big overlap conn}
    \[ \sum_{\substack{F \subset \Cnlr, v(F) = n, F \text{ conn.}}} \frac{\Prob{\C \cap \CC \cong F}}{p^{e(F)}} \lesssim_{r,\ell} \frac{1}{\E{Z(\Cnlr)}}. \]
\end{lemma}
\begin{proof}
    We now handle the terms that have a large number of vertices, and show that they contribute negligibly. We start with the case that $F$ is a single connected spanning component. For this proof we write $m = \frac{n}{r-\ell}$ for the number of edges in $\Cnlr$. To start we have the bound
    \begin{align*}
        \Prob{\C \cap \CC \cong F}p^{-e(F)} &\leq \Prob{\C \cap \CC \cong \Cnlr}m(r!)^{m-e(F)} p^{-m}p^{m-e(F)} \leq \frac{m(r!)^{m-e(F)} p^{m-e(F)}}{\E{Z(\Cnlr)}}.
    \end{align*}
    The term $m$ comes from the choice of which copy of $F$ in $\C$ to overlap. Given this choice of $F$, there are at most $(r!)^{m-e(F)}$ cycles $\CC$ that contain it, coming from arbitrarily permuting the vertices in each of the missing edges. Moreover, there are at most $\binom{m-1}{k-1}$ choices of $F$ such that $m - e(F) = k$. Thus, the total contribution of all connected spanning subgraphs $F$ is bounded by 
    \begin{align*}
        \sum_{\substack{F \subset \Cnlr \\ v(F) = n \\ F \text{ conn.}}} \frac{\Prob{\C \cap \CC \cong F}}{p^{e(F)}} &\leq \frac{1}{\E{Z(\Cnlr)}}\left(1 + \sum_{k=1}^m \binom{m-1}{k-1}m(r!)^{k}p^{k} \right) \\
        &= \frac{1}{\E{Z(\Cnlr)}}\left(1 + \frac{ct!(s-t)!(r!)e^{r-\ell}}{(r-\ell)n^{r-\ell-1}}\sum_{k=1}^m \binom{m-1}{k-1}(r!)^{k-1}p^{k-1} \right) \\
        &= \frac{1}{\E{Z(\Cnlr)}}\left(1 + \frac{ct!(s-t)!(r!)e^{r-\ell}}{(r-\ell)n^{r-\ell-1}}(1 + r!p)^{m-1} \right) \\
        &\leq \frac{1}{\E{Z(\Cnlr)}}\left(1 + \frac{ct!(s-t)!(r!)e^{r-\ell}}{(r-\ell)n^{r-\ell-1}}e^{r!p(m-1)} \right) \\
        &\lesssim_{r,\ell} \frac{1}{\E{Z(\Cnlr)}}.
    \end{align*}
    This suffices to prove the claim, but we give here a stronger bound in the case $r-\ell=1$ for future utility. The key observation is that in the case $r-\ell=1$, we cannot have two Hamilton $\ell$-cycles that overlap in $\Cnlr$ minus a single edge. Indeed, for any pair of vertices there are at least two edges separating them, so any swap of vertices impacts at least two edges. By this observation, there are at most $\binom{m-1}{k-2}$ ways to choose an $F$ missing $k$ edges rather than $\binom{m-1}{k-1}$. Using this new bound we get
    \begin{align*}
        \sum_{\substack{F \subset \Cnlr \\ v(F) = n \\ F \text{ conn.}}} \frac{\Prob{\C \cap \CC \cong F}}{p^{e(F)}} &\leq \frac{1}{\E{Z(\Cnlr)}}\left(1 + \sum_{k=1}^m \binom{m-1}{k-2}m(r!)^{k}p^{k} \right) \\
        &= \frac{1}{\E{Z(\Cnlr)}}\left(1 + \frac{(ct!(s-t)!r!e)^2}{(r-\ell)n}\sum_{k=1}^m \binom{m-1}{k-2}(r!)^{k-2}p^{k-2} \right) \\
        &\leq \frac{1}{\E{Z(\Cnlr)}}\left(1 + O(\frac{1}{n})\right).
    \end{align*}
\end{proof}
\begin{lemma}\label{lemma: big overlap disconn}
    \[ \sum_{v(F) = n, F \text{ disconn.}} \frac{\E{N_F(\C \cap \CC)}}{p^{e(F)}} \lesssim_{r,\ell} \frac{1}{n^2}. \]
\end{lemma}
\begin{proof}
    The only case that remains is when $F$ consists of $n$ vertices but multiple components. To handle this, we exploit the fact that in this case, the factorization (\ref{eq: decomp}) is off by a factor of at least $n^2$. Indeed, when choosing the locations of each component, since the subgraph covers all of the vertices, the last component is fixed and has no choices left. Thus, we gain a factor of $n^2$ over the above estimate, and so 
    \[ \sum_{v(F) = n, F \text{ disconn.}} \frac{\E{N_F(\C \cap \CC)}}{p^{e(F)}} \lesssim_{r,\ell} \frac{1}{n^2}. \]
\end{proof}
Summing the inequalities yields the following (sharp) estimate on the second moment of $Z(\Cnlr)$. We will however need to use the refined cases above in the next section. 
\begin{proposition}\label{prop: second moment}
    When $\E{Z(\Cnlr)} \to \infty$, 
    \[ \frac{\E{Z(\Cnlr)^2}}{\E{Z(\Cnlr)}^2} = \sum_{t=1}^{m} \frac{\Prob{|\C \cap \CC| = t}}{p^t} \leq \exp((1+o(1))n^{2-\ell}\sum_{k=1}^{\log n} \frac{A_k c^{-k} e^{-k(r-\ell)}}{(r-\ell)^2}). \]
\end{proposition}
Notice that when $\ell \geq 3$, the upper bound on the right hand side is $1+o(1)$. A standard application of Chebyshev's inequality proves part 1 of Theorem \ref{thm: lognormal}.

\section{Small subgraph conditioning}\label{sec: SSC}
In this section, we set up the small subgraph conditioning and explain the improvement over the vanilla second moment. Since this improvement is only needed in the case $\ell = 2$, we will assume this for the next two sections. The main claim is the following lemma. 
\begin{lemma}\label{lemma: second moment}
    \[ \frac{\E{X^2}}{\E{X}^2} = \sum_{t=1}^{m} \frac{\Prob{|\C \cap \CC| = t}}{p^t} \cdot \frac{\EE_t^{*2}[e^{-2Y}]}{\EE^*[e^{-Y}]^2}. \]
\end{lemma}
Notice that compared to the vanilla second moment on $Z(\Cnlr)$, the right hand side has an additional term defined by the moment generating function of $Y$ under a planted distribution. As will be shown later in Lemma \ref{lemma: CLT}, the additional term contributes a small (less than 1) constant factor for $t$ small, and a large constant factor for $t$ large. The key to the argument is that the small constant factor exactly cancels the extra constant factor in the vanilla second moment, and the large constant factor is negligible compared to the probability of large overlap.
\begin{proof}
    Let $p(G) = \Prob{G_r(n,p) = G} = p^{e(G)}(1-p)^{\binom{n}{r}-e(G)}$. We write 
    \begin{align*}
        \E{Z(\Cnlr)e^{-Y}} &= \sum_G p(G) Z(\Cnlr)e^{-Y(G)} = \sum_G \sum_{C_1} p(G) \mathbbm{1}\{C_1 \subset G\} e^{-Y(G)}
    \end{align*}
    where the first sum is over all graphs $G$. Note that 
    \[ \EE^*[e^{-Y}] = \frac{1}{N_C} \sum_{C_1} \frac{\sum_G \mathbbm{1}\{C_1 \subset G\} p(G) e^{-Y(G)}}{\sum_G \mathbbm{1}\{C_1 \subset G\} p(G)} = \frac{1}{\E{N_C(G)}} \sum_{C_1} \sum_G \mathbbm{1}\{C_1 \subset G\}p(G) e^{-Y(G)}. \]
    This implies 
    \[ \E{Z(\Cnlr)e^{-Y}} = \E{Z(\Cnlr)}\EE^*[e^{-Y}]. \]
    We perform a similar transformation for the second moment. 
    \begin{align*}
        \E{Z^2e^{-2Y}} &= \sum_G \sum_{C_1, C_2} p(G)\mathbbm{1}\{C_1 \cup C_2 \subset G\}e^{-2Y(G)} \\ 
        &= \sum_{k = 0}^{e(\Cnlr)} \sum_{e(C_1 \cap C_2) = k} \sum_G p(G) \mathbbm{1}\{C_1 \cup C_2 \subset G\}e^{-2Y(G)}.
    \end{align*}
    This time, letting $N_C^{(2)}(k)$ be the number of pairs of copies of $\Cnlr$ which intersect at $k$ edges, we have 
    \begin{align*}
        \EE_k^{*2}[e^{-2Y}] &= \frac{1}{N_C^{(2)}(k)} \sum_{e(C_1 \cap C_2) = k} \frac{ \sum_G p(G) \mathbbm{1}\{C_1 \cup C_2 \subset G\} e^{-2Y(G)} }{\sum_G p(G) \mathbbm{1}\{C_1 \cup C_2 \subset G\}} \\
        &= \frac{1}{N_C^2 \Prob{e(C_1 \cap C_2) = k} p^{2e(C) - k}} \sum_{e(C_1 \cap C_2) = k}\sum_G p(G) \mathbbm{1}\{C_1 \cup C_2 \subset G\}e^{-2Y(G)} \\
        &= \frac{1}{\E{N_C(G)}^2\Prob{e(C_1 \cap C_2) = k} p^{-k}} \sum_{e(C_1 \cap C_2) = k}\sum_G p(G) \mathbbm{1}\{C_1 \cup C_2 \subset G\}e^{-2Y(G)}
    \end{align*}
    Substituting this in yields
    \begin{align*}
        \E{Z(\Cnlr)^2e^{-2Y}} = \E{Z(\Cnlr)}^2 \sum_{k=0}^{e(\Cnlr)} \Prob{e(\C \cap \CC) = k}p^{-k}\EE_k^{*2}[e^{-2Y}]
    \end{align*}
    The conclusion follows by plugging in the first moment expression. 
\end{proof}

Following Lemma \ref{lemma: second moment}, to complete the ``conditioning'' portion of the method we need to evaluate the moment generating function of $Y$ with respect to various planted distributions. The next lemma shows that the variable $Y$ converges to a Gaussian random variable, and thus the moment generating functions will be easy to evaluate. 
\begin{lemma}\label{lemma: CLT}
    Under the null model $\PP$, for any integer $k$ as $n \to \infty$, 
    \[ (Y(P_1), Y(P_2), \ldots, Y(P_k)) \to \mathcal{N}(0, I_k). \]
    Under the planted model $\PP^*$, for any integer $k$ as $n \to \infty$, 
    \[ (Y(P_1), Y(P_2), \ldots, Y(P_k)) \to \mathcal{N}(\mu, I_k) \]
    where $\mu_j = \ \frac{\sqrt{A_j c^{-j} e^{-j(r-\ell)}}}{(r-\ell)}$.
    Under the double planted model $\PP_t^{*2}$, for any integer $k$ and $t = o(n)$ as $n \to \infty$, 
    \[ (Y(P_1), Y(P_2), \ldots, Y(P_k)) \to \mathcal{N}(2\mu, I_k). \]
\end{lemma}
\begin{proof}
    The convergence to a Gaussian distribution under the null model follows from work of Janson using orthogonal decompositions \cite[Theorem 3, Remark 5.1]{J:94}. Note here that we have defined each $Y(P_k)$ in orthogonal decomposition form, and every component is connected. We now compute the mean and variances. Recall that for any $j$, we have 
    \[ Y(P_j) = \frac{1}{\sqrt{N_{P_j}}}\sum_{P_j} \prod_{e \in P_j} Y_e \]
    where $Y_e = \frac{\mathbbm{1}\{e \in G\} - p}{\sqrt{p(1-p)}}$ and $\E{Y(P_j)} = 0$ and $\Var{Y(P_j)} = 1$. Notice that for any $j_1 \neq j_2$, and each copy of $P_{j_1}$ and $P_{j_2}$ we have $\E{\prod_{e_1 \in P_{j_1}} Y_{e_1} \prod_{e_2 \in P_{j_2}} Y_{e_2}} = 0$. This implies that $\E{Y(P_{j_1})Y(P_{j_2})} = 0$ so the statistics for distinct subgraphs are uncorrelated. Thus, since each variable converges to a standard Gaussian, their joint limit is a standard $k$-dimensional Gaussian.

    We now study the behavior of the random variables under the planted model. The convergence to a Gaussian distribution is a little more subtle. We can view the planted model as sampling a random subgraph of $K_n \setminus C_n$. The copies of $P_j$ that have planted edges now have coefficients $\sqrt{\frac{1-p}{p}}$ raised to the number of planted edges. Now the general orthogonal decomposition outlined by Janson no longer holds, as we distinguish equivalence classes only up to isomorphism of $K_n \setminus C_n$. Fortunately for us, we have written the definition of $Y(P_j)$ in terms of an orthogonal decomposition, and importantly the coefficients of all isomorphic subgraphs are the same -- it is determined by the number of edges missing from $P_j$. Thus, the convergence follows once we check that the dominant contribution comes from connected subgraphs.
    
    We compute the mean and covariance structure of the resulting Gaussian. From this computation, it will follows that the dominant contribution comes from the full $P_j$'s with no planted edges. We can compute 
    \begin{align*}
        \EE^*[Y(P_j)] &= \EE^*\left[ \frac{1}{\sqrt{N_{P_j}}}\sum_{P_j} \prod_{e \in P_j} Y_e \right] = \frac{1}{\sqrt{N_{P_j}}}\sum_{P_j} \prod_{e \in P_j} \EE^*[Y_e] = \frac{n}{(r-\ell)} \cdot \sqrt{\frac{(1-p)^j}{p^jN_{P_j}}} \\
        &= (1+o(1))\frac{n}{r-\ell}\cdot\left(\frac{n^{(r-\ell)}}{ce^{(r-\ell)}\lambda}\right)^{j/2}\cdot \sqrt{\frac{\mathrm{Aut}(P_j)}{(n)_{v(P_j)}}}  = (1+o(1))\frac{\sqrt{A_j c^{-j} e^{-j(r-\ell)}}}{(r-\ell)}.
    \end{align*}
    The last equality follows since $\EE^*[Y_e] = \sqrt{\frac{1-p}{p}}$ if $e$ is a planted edge and $\EE^*[Y_e] = 0$ if $e$ is not a planted edge. The number of copies of $P_j$ which consist only of planted edges is the number of copies of $P_j$ in $\Cnlr$, which is $\frac{n}{r-\ell}$. 
    \begin{align*}
        \mathrm{Var}^*(Y(P_j)) &= \EE^*[Y(P_j)^2] - \EE^*[Y(P_j)]^2 \\
        &= \frac{1}{N_{P_j}}\sum_{P_1, P_2} \EE^*\left[\prod_{e_1 \in P_1} Y_{e_1} \prod_{e_2 \in P_2} Y_{e_2}\right] - \EE^*\left[\prod_{e_1 \in P_1} Y_{e_1}\right] \EE^*\left[\prod_{e_2 \in P_2} Y_{e_2}\right] \\
        &= \frac{1}{N_{P_j}} \sum_{P_1, P_2} \left(\frac{1-p}{p}\right)^{e(P_1 \cap C^*)}\mathbbm{1}\{P_1 \triangle P_2 \subset C^*, P_1 \cap P_2 \not\subset C^*\}.
    \end{align*}
    For each pair where $P_1 = P_2$ and $P_1 \cap C^* = \emptyset$, we get a contribution of 1. This amounts to a total contribution of $1-o(1)$. We show that all other terms are negligible. Let $(P_1, P_2)$ be such that $e(P_1 \cap C^*) = i \neq 0$, $P_1 \triangle P_2 \subset C^*$ and $P_1 \cap P_2 \not\subset C^*$. This pair contributes $\left( \frac{1-p}{p}\right)^i = (1+o(1))p^{-i}$. We now proceed to bound the number of such pairs. Let $a$ denote the number of $\ell$-path segments in $P_1 \cap C^*$. By stars and bars, we can choose the sizes of these segments in $\binom{i+a-1}{i}$ ways. We choose the location of each segment along $C^*$ in $\frac{n}{r-\ell}$ ways, leading to at most $\left(\frac{n}{r-\ell}\right)^a$ choices. There are $(r-\ell)(j-i) + \ell - \ell a$ vertices in $P_1$ that have yet to be fixed, leading to $n^{(r-\ell)(j-i) + \ell - \ell a}$ choices. Finally, given $P_1$, $P_2$ must live in the union $P_1 \cup C^*$. There are at most $ja2^a$ possibilities -- $j$ choices for which edge of $P_2$ first meets $C^*$, $a$ choices for where $P_2$ meets $C^*$, and $2$ choices for which structure to follow at each of the $a$ branching points. All together, this is an upper bound of
    \[ \sum_{a=1}^{j} \binom{i+a-1}{i}\left(\frac{n}{r-\ell}\right)^{a} n^{(r-\ell)(j-i)+\ell-\ell a} \cdot ja2^a\lesssim_{r,\ell} n^{(r-\ell)(j-i)+1}. \]
    Thus, the total contribution from such pairs is at most 
    \[ \frac{1}{N_{P_j}}n^{(r-\ell)(j-i)+1}p^{-i} \lesssim_{r,\ell} n^{1-\ell}. \]
    Summing from $i = 1$ to $j$, we obtain the desired estimate that 
    \[ \mathrm{Var}^*(Y(P_j)) = 1+o(1). \]
    For the covariance, we see that in computing $\EE^*[Y(P_{j_1})Y(P_{j_2})]$ we no longer have the possibility that $P_1 = P_2$ and therefore there is no contribution of 1. The same argument shows that $\mathrm{Cov}^*(Y(P_{j_1}),Y(P_{j_2})) = o(1)$. 
    
    Similarly, under the double planted model $\PP_t^{*2}$ the convergence to standard Gaussians follows by observing that all equivalence classes of isomorphic graphs have the same coefficients. By the calculation in the planted case, each of the two cycles contributes $\frac{\sqrt{A_j c^{-j} e^{-j(r-\ell)}}}{(r-\ell)}$ to the expectation. It remains to estimate the contributions involving both cycles. There are $t$ edges at which the cycles intersect, and at most $j2^j$ paths emanating from each of these intersections. Thus, the total contribution from paths intersecting both cycles is bounded by 
    \[ tj2^j\sqrt{\frac{(1-p)^j}{p^jN_{P_j}}} = o(1). \]
    Thus all together the expectation is 
    \begin{align*}
        \EE_t^{*2}[Y(P_j)]
        &= (1+o(1))\frac{2\sqrt{A_j c^{-j} e^{-j(r-\ell)}}}{(r-\ell)}.
    \end{align*}
    The covariance structure follows from a similar calculation to the planted model case. We again count the number of pairs $P_1$ and $P_2$ such that $e(P_1 \cap (C_1^* \cup C_2^*)) = i$, $P_1 \triangle P_2 \subset C_1^* \cup C_2^*$ and $P_1 \cap P_2 \not \subset C_1^* \cup C_2^*$. Once again let $a$ denote the number of $\ell$-path segments in $P_1 \cap (C_1^* \cup C_2^*)$. The number of ways to choose the sizes of these segments remains $\binom{i+a-1}{i}$. We can choose the locations of the segments in $\left(\frac{2n}{r-\ell}\right)^a$ ways, but these segments are no longer uniquely determined as they can oscillate between $C_1^*$ and $C_2^*$. This can happen in at most $2^i$ ways. The number of ways to choose the remaining vertices of $P_1$ is still $n^{(r-\ell)(j-i)+\ell-\ell a}$. Finally, $P_2$ must live in the union $P_1 \cup C_1^* \cup C_2^*$. There are $j$ choices for which edge of $P_2$ meets $C_1^* \cup C_2^*$, $a$ choices for the meeting location, and $3$ choices at each edge for which path to follow. Thus, there are at most $ja3^j$ choices for $P_2$. All together, we obtain an upper bound of 
    \[ \sum_{a=1}^j \binom{i+a-1}{i}\left(\frac{2n}{r-\ell}\right)^an^{(r-\ell)(j-i)+\ell-\ell a}ja3^j \lesssim_{r,\ell} n^{(r-\ell)(j-i)+1} \]
    as before. Thus, the total contribution remains $n^{1-\ell}$, and we have shown that $\mathrm{Var}_t^{*2}(Y(P_j)) = 1+o(1)$, and $\mathrm{Cov}_t^{*2}(Y(P_{j_1}), Y(P_{j_2})) = o(1)$. 
\end{proof}

Using the above central limit theorem, we evaluate the additional multiplicative terms from the right hand side of Lemma \ref{lemma: second moment}.
\begin{lemma}\label{lemma: CLTcor}
    Let $Y_N = \sum_{k=1}^N t_kY(P_k)$. For all $t = o(n)$ as $n \to \infty$, 
    \[ \frac{\EE_t^{*2}[e^{-2Y_N}]}{\EE^*[e^{-Y_N}]^2} = \exp(-(1+o(1))\sum_{k=1}^{N} \frac{A_k c^{-k} e^{-k(r-\ell)}}{(r-\ell)^2}). \]
\end{lemma}
\begin{proof}
    Using Lemma \ref{lemma: CLT}, we can compute
    \begin{align*}
        \EE^*[e^{-Y_N}] &= \exp(\sum_{k=1}^{N} -t_k\mu_k + \frac{t_k^2}{2}) = \exp(-(1+o(1))\sum_{k=1}^{N} \frac{A_k c^{-k}e^{-k(r-\ell)}}{2(r-\ell)^2}) \\
        \EE_t^{*2}[e^{-2Y_N}] &= \exp(\sum_{k=1}^{N} -2t_k(2\mu_k) + 2t_k^2) = \exp(-(1+o(1))2\sum_{k=1}^{N} \frac{A_k c^{-k}e^{-k(r-\ell)}}{(r-\ell)^2})
    \end{align*}
    The convergence follows. 
\end{proof}

\section{Proof of Theorem \ref{thm: lognormal}}\label{sec: PMT}
In this section, we combine the results of Sections \ref{sec: SMC} and \ref{sec: SSC} to prove Theorem \ref{thm: lognormal} in the case $\ell = 2$. To begin, we use the sharp second moment to deduce that $X$ converges in probability to a constant.
\begin{lemma}\label{lemma: plim}
    \[ \frac{Z(\Cnlr)}{\E{Z(\Cnlr)}e^Y} \overset{\mathrm{p}}{\longrightarrow} \exp(-\sum_{k=1}^{\infty} \frac{A_k c^{-k} e^{-k(r-\ell)}}{2(r-\ell)^2}). \]
\end{lemma}
\begin{proof}
    Fix $\epsilon > 0$. For any $\delta > 0$, choose $N = N_0(\epsilon, \delta)$ such that $\exp(\sum_{k=N_0}^\infty \frac{A_k c^{-k} e^{-k(r-\ell)}}{(r-\ell)^2}) < 1+\epsilon^2\delta$ and $e^{-N(r-\ell)} \leq \delta(\log (1+\epsilon))^2$.
    
    Denote the desired limit by $L = \exp(-\sum_{k=1}^{\infty} \frac{A_k c^{-k} e^{-k(r-\ell)}}{2(r-\ell)^2})$. We write 
    \[  \abs{X - L} \leq |X - \tilde X| + |\tilde X - \tilde X_N| + |\tilde X_N - L| \]
    where $\tilde X$ be the truncation $\tilde X = \frac{Z(\Cnlr)}{\E{Z(\Cnlr)}}\exp(-\tilde Y)$ with $\tilde Y = \max\{ \min\{ Y, M\}, -M\}$ and $M = \min\{\log\log \E{Z(\Cnlr)}, \log\log n\}$. Recall that $Y_N$ denotes the partial sum $\sum_{k=1}^N t_kY(P_k)$, and define the truncation $\tilde X_N$ similarly.

    We first control the difference between $Y$ and $Y_N$. Notice that 
    \[ \E{(Y - Y_N)^2} = \E{\left(\sum_{k=N}^{\log n} t_kY(P_k)\right)^2} = \sum_{k=N}^{\log n} t_k^2\E{Y(P_k)^2} = \sum_{k=N}^{\log n} \frac{A_k c^{-k} e^{-k(r-\ell)}}{(r-\ell)^2} \lesssim_{r,\ell} e^{-N(r-\ell)}. \]

    The first difference is small since $Y$ is close to $Y_N$, and Lemma \ref{lemma: CLT} implies $Y_N$ is normally distributed, and thus is very large in magnitude with small probability. Indeed for any $\epsilon > 0$, 
    \[ \Prob{|X - \tilde X| > \epsilon} \leq \Prob{Y \neq \tilde Y} = \Prob{|Y| > M} \leq \Prob{|Y_N| > \frac{M}{2}} + \Prob{|Y - Y_N| > \frac{M}{2}} = o(1). \]

    For the third difference, recall from (a variant of) Lemma \ref{lemma: second moment} that we can write
    \[ \frac{\EE[\tilde X_N^2]}{\EE[\tilde X_N]^2} = \sum_{t=1}^{m} \frac{\Prob{|\C \cap \CC| = t}}{p^t} \cdot \frac{\EE_t^{*2}[e^{-2\tilde Y_N}]}{\EE^*[e^{-\tilde Y_N}]^2} = \sum_{F \subset \Cnlr} \frac{\E{N_F(\C \cap \CC)}}{p^{e(F)}}\frac{\EE_{e(F)}^{*2}[e^{-2\tilde Y_N}]}{\EE^*[e^{-\tilde Y_N}]^2}. \]
    We once again handle the various subgraphs in a few cases, similarly to the analysis in Section \ref{sec: SMC}. \\
        \textbf{Case 1.} If $v(F) \leq \log n$ then Lemmas \ref{lemma: small overlap} and \ref{lemma: CLTcor} yield that
        \begin{multline*}
            \sum_{F \subset \Cnlr, v(F) \leq \log n} \frac{\E{N_F(\C \cap \CC)}}{p^{e(F)}}\frac{\EE_{e(F)}^{*2}[e^{-2\tilde Y_N}]}{\EE^*[e^{-\tilde Y_N}]^2} \leq \\ \exp((1+o(1))\sum_{k=1}^{\log n} \frac{A_k}{(r-\ell)^2}c^{-k}e^{-k(r-\ell)}) \cdot \\ 
            \exp(-(1+o(1))\sum_{k=1}^{N} \frac{A_k}{(r-\ell)^2}c^{-k}e^{-k(r-\ell)}) \leq 1 + \epsilon^2\delta + o(1).
        \end{multline*}
        Here we used that truncating $Y_N$ to $\tilde Y_N$ only changes the first and second moments by $o(1)$ factors.
        \textbf{Case 2.} If $\log n < v(F) < n$ then Lemma \ref{lemma: mid overlap} along with the truncation implies 
        \[ \sum_{F \subset \Cnlr, v(F) \leq \log n} \frac{\E{N_F(\C \cap \CC)}}{p^{e(F)}}\frac{\EE_{e(F)}^{*2}[e^{-2\tilde Y_N}]}{\EE^*[e^{-\tilde Y_N}]^2} \lesssim_{r,\ell} \frac{e^{4M}}{n} \leq \frac{(\log n)^4}{n} = o(1). \]
        \textbf{Case 3.} If $v(F) = n$ and $F$ is connected, then Lemma \ref{lemma: big overlap conn} along with the truncation implies 
        \[ \sum_{F \subset \Cnlr, v(F)=n, F\text{ conn.}} \frac{\Prob{\C \cap \CC \cong F}}{p^{e(F)}}\frac{\EE_{e(F)}^{*2}[e^{-2\tilde Y_N}]}{\EE^*[e^{-\tilde Y_N}]^2} \lesssim_{r,\ell} \frac{e^{4M}}{\E{Z(\Cnlr)}} \leq \frac{\left(\log \E{Z(\Cnlr)}\right)^4}{\E{Z(\Cnlr)}}. \]
        \textbf{Case 4.} If $v(F) = n$ and $F$ is disconnected then Lemma \ref{lemma: big overlap disconn} along with the truncation implies 
        \[ \sum_{F \subset \Cnlr, v(F)=n, F\text{ disconn.}} \frac{\E{N_F(\C \cap \CC)}}{p^{e(F)}}\frac{\EE_{e(F)}^{*2}[e^{-2\tilde Y_N}]}{\EE^*[e^{-\tilde Y_N}]^2} \lesssim_{r,\ell} \frac{e^{4M}}{n^2} \leq \frac{(\log n)^4}{n^2} = o(1). \]
    In particular, when $\E{Z(\Cnlr)} \to \infty$, we have that 
    \[ \frac{\EE[\tilde X_N^2]}{\EE[\tilde X_N]^2} \leq 1 + \epsilon^2\delta + o(1). \]
    By an application of Chebyshev's inequality we deduce
    \[ \Prob{\abs{\tilde X_N - \EE[\tilde X_N]} > \epsilon \EE[\tilde X_N]} \leq \delta + o(1). \]
    Note from the proof of Lemma \ref{lemma: CLTcor} that 
    \[ \EE[\tilde X_N] = \E{\frac{Z(\Cnlr)}{\E{Z(\Cnlr)}e^{\tilde Y_N}}} = \EE^*[e^{-{\tilde Y_N}}] = (1+o(1))\exp(-\sum_{k=1}^{N} \frac{A_k c^{-k} e^{-k(r-\ell)}}{2(r-\ell)^2}) \]
    which is bounded by a constant in $r$ and $\ell$. Since 
    \begin{multline*}
        \exp(-\sum_{k=1}^{N} \frac{A_k c^{-k} e^{-k(r-\ell)}}{2(r-\ell)^2}) - L = \\ \exp(-\sum_{k=1}^{N} \frac{A_k c^{-k} e^{-k(r-\ell)}}{2(r-\ell)^2})\left(1 - \exp(\sum_{k=N}^{\infty} \frac{A_k c^{-k} e^{-k(r-\ell)}}{2(r-\ell)^2})\right) \lesssim_{r,\ell} \epsilon^2\delta,
    \end{multline*}
    up to constant factors the third difference is greater than $\epsilon$ with probability at most $\delta$. 

    For the middle difference, we estimate the probability 
    \[ \Prob{\exp(|\tilde Y - \tilde Y_N|) > 1+\epsilon}. \]
    We decompose this event into three cases: either one of $Y$ or $Y_N$ is large, or the difference between $Y$ and $Y_N$ is large. 
    \begin{align*}
        \Prob{\exp(|\tilde Y - \tilde Y_N|) > 1+\epsilon} &\leq \Prob{Y \neq \tilde Y} + \Prob{Y_N \neq \tilde Y_N} + \Prob{|Y - Y_N| > \log(1+\epsilon)} \\
        &\lesssim_{r,\ell} \frac{e^{-N(r-\ell)}}{(\log(1+\epsilon))^2} + o(1) \leq \delta + o(1).
    \end{align*}
    This implies that 
    \[ \Prob{|\tilde X - \tilde X_N| > \epsilon \tilde X_N} \leq \delta + o(1). \]
    Recall from above that with probability at least $1 - \delta$, $\tilde X_N$ is close to $\EE[\tilde X_N]$ which is bounded by a constant. This yields 
    \[ \Prob{|\tilde X - \tilde X_N| > \epsilon \EE[\tilde X_N]} \leq 2\delta + o(1) \]
    which is the desired bound on the middle difference.
    
    Combining the estimates on all three terms, and since $\delta > 0$ was chosen arbitrarily, the claimed convergence follows. 
\end{proof}
\begin{proof}[Proof of Theorem \ref{thm: lognormal}]
    By Lemma \ref{lemma: CLT}
    \[ Y_N \overset{\mathrm{d}}{\longrightarrow} \mathcal{N}\left(0, \sum_{k=1}^{N} \frac{A_k c^{-k} e^{-k(r-\ell)}}{(r-\ell)^2} \right). \]
    We also have from the proof of Lemma \ref{lemma: plim} that $\E{(Y - Y_N)^2} \lesssim_{r,\ell} e^{-N(r-\ell)}$, i.e. $Y_N$ converges to $Y$ in $L^2$. Thus, sending $N \to \infty$ we deduce that 
    \[ Y \overset{\mathrm{d}}{\longrightarrow} \mathcal{N}\left(0, \sum_{k=1}^{\infty} \frac{A_k c^{-k} e^{-k(r-\ell)}}{(r-\ell)^2} \right). \]
    Thus, we know that 
    \[ e^Y \overset{\mathrm{d}}{\longrightarrow} \mathrm{Lognormal}\left(0, \sum_{k=1}^{\infty} \frac{A_k c^{-k} e^{-k(r-\ell)}}{(r-\ell)^2} \right). \]
    The theorem follows from applying Lemma \ref{lemma: plim} and noting that multiplying by an exponential factor shifts the mean. 
\end{proof}

\section{Proof of Theorem \ref{thm: poisson}}\label{sec: poisson}
In this section we prove Theorem \ref{thm: poisson} describing the constant expectation regime. The proof relies on a sub-sampling scheme based on the following heuristic. Suppose we sample a host hypergraph with a density that induces a slowly growing number of Hamilton $\ell$-cycles in expectation. Theorem \ref{thm: lognormal} gives us good control over the actual number of cycles that appear. The next lemma -- Lemma \ref{lemma: no big overlap} -- implies a near independence in the sense that the cycles are pairwise nearly disjoint with high probability. Thus, if we sub-sample the edges of this host hypergraph at the appropriate rate, each cycle should be retained approximately independently, inducing a Poisson type behavior.

We begin with the approximate independence lemma, which follows from our second moment computation combined with a basic application of Markov's inequality.
\begin{lemma}\label{lemma: no big overlap}
    Let $\ell \geq 2$ and $p = p(n)$ be such that $\E{Z(\Cnlr)} = \log n$. With probability $1-o(1)$, there do not exist $C_1$ and $C_2$ Hamilton $\ell$-cycles in $G_r(n,p)$ such that $\log n \leq |C_1 \cap C_2| \leq \frac{n}{r-\ell}-1$.
\end{lemma}
\begin{proof}
    We compute the expected number of such pairs of cycles. For any fixed $t$, by linearity we can express the expected number of pairs of cycles with overlap $t$ as 
    \[ N_{\Cnlr}^2\Prob{|C_1 \cap C_2| = t}p^{2\frac{n}{r-\ell}-t} = \log^2 n \cdot \frac{\Prob{|C_1 \cap C_2| = t}}{p^t}. \]
    By Lemmas \ref{lemma: mid overlap}, \ref{lemma: big overlap conn} and \ref{lemma: big overlap disconn} we know 
    \[ \sum_{\log n < t < \frac{n}{r-\ell}-1} \frac{\Prob{|C_1 \cap C_2| = t}}{p^t} \lesssim_{r,\ell} \frac{1}{n}. \]
    Note here that we do not have the $\frac{1}{\E{Z(\Cnlr)}}$ term from Lemma \ref{lemma: big overlap conn} as this contribution arises from full overlap which is excluded in this sum. A quick scan through the proof reveals that the remaining terms are all $O(\frac{1}{n})$. Thus, the expected number of pairs of cycles that have overlap between $\log n$ and $\frac{n}{r-\ell}-1$ is bounded by 
    \[ \sum_{\log n < t < \frac{n}{r-\ell}-1} N^2_{\Cnlr}\Prob{|C_1 \cap C_2| = t}p^{2\frac{n}{r-\ell}-t} \lesssim_{r,\ell} \frac{\log^2 n}{n}. \]
    The claim then follows by Markov's inequality.
\end{proof}
\begin{proof}[Proof of Theorem \ref{thm: poisson}]
    Generate a random hypergraph $G \sim G_r(n,p)$ as follows. First sample $G' \sim G_r(n, p')$ where $p' = p'(n)$ is such that $\E{Z(\Cnlr)} = \log n$. To obtain $G$ from $G'$, retain each edge independently with probability $\left(\frac{m}{\log n}\right)^{\frac{r-\ell}{n}}$. 
    
    Let the number of Hamilton $\ell$-cycles in $G'$ be denoted by $Z'$. By Lemma \ref{lemma: no big overlap}, all pairs of cycles in $G'$ overlap in at most $\log n$ edges with high probability. On this event, there are at most $(Z')^2\log n$ edges that belong to multiple cycles. Notice that the deletion probability of an edge is 
    \[ 1 - \left(\frac{m}{\log n}\right)^{\frac{r-\ell}{n}} = O\left(\frac{\log\log n}{n}\right) \]
    so with probability at least $1 - O(\frac{(Z')^2\log n \log\log n}{n})$ no edge belonging to multiple cycles will be deleted in our sub-sampling procedure. By Theorem \ref{thm: lognormal} this is $1-o(1)$ with high probability. On the event that no edge belonging to multiple cycles is deleted, each cycle is retained independently with probability $\frac{m}{\log n}$. It follows that the number of cycles in $G$ differs in total variation distance from $\mathrm{Bin}(Z', \frac{m}{\log n})$ by $o(1)$. We now relate this to the corresponding Poisson distribution. By Le Cam's theorem, 
    \[d_{\mathrm{TV}}\left(\mathrm{Bin}\left(Z', \frac{m}{\log n}\right), \mathrm{Pois}\left(m\cdot \frac{Z'}{\log n}\right)\right) \leq Z'\cdot \left(\frac{m}{\log n} \right)^2. \]
    Once again by Theorem \ref{thm: lognormal}, we know that with high probability the right hand side is $o(1)$. Combining this with the fact that for $\ell \geq 3$, $\frac{Z'}{\log n} = 1+o(1)$ with high probability suffices to prove part 1. For part 2, we check the convergence of the cumulative distribution function.
    \begin{multline*}
        \Prob{\mathrm{Pois}(m\cdot \frac{Z'}{\log n}) \leq k} = \int \Prob{\mathrm{Pois}(m\cdot z) \leq k} \ d\Prob{\frac{Z'}{\log n} \leq z} \\
        \xrightarrow{n\to\infty} \int\Prob{\mathrm{Pois}(m\cdot z) \leq k} \ d\Prob{L \leq z} = \Prob{\mathrm{Pois}(m\cdot L) \leq k}
    \end{multline*}
    where we used the fact that $\Prob{\mathrm{Pois}(m \cdot z) \leq k}$ is a bounded continuous function of $z$ and the distributional convergence of $\frac{Z'}{\log n}$ to $L$. 
\end{proof}

\bibliographystyle{amsplain0}
\bibliography{hamilton}

\end{document}